\documentclass{article}

\usepackage{graphicx}
\usepackage{amsfonts, amsmath, amssymb, amsthm}
\usepackage{authblk}
\usepackage{cite}
\usepackage{float}
\usepackage{subcaption}
\usepackage[T1]{fontenc}

\usepackage{xcolor}
\usepackage{xspace}

\usepackage{hyperref}
\usepackage{cleveref}

\usepackage[textwidth=1.5in]{todonotes}
\usepackage[letterpaper, margin=1in, includefoot]{geometry}

\usepackage{thmtools}
\usepackage{thm-restate}

\usepackage{tikz}
\usetikzlibrary{positioning}
\usetikzlibrary{calc}
\usetikzlibrary{math}

\tikzset{
	position/.style args={#1:#2 from #3}{
		at=($(#3)+(#1:#2)$)
	}
}

\tikzset{
  v:main/.style = {draw, circle, scale=0.8, thick,fill=black,inner sep=0.7mm},
  v:ghost/.style = {inner sep=0pt,scale=1},
  >={latex},
  e:marker/.style = {line width=8.5pt,line cap=round,opacity=0.35,color=DarkGoldenrod},
  e:main/.style = {line width=1pt},
}

\newcommand{\mis}{\textsc{Maximum Weighted Independent Set}\xspace}
\newcommand{\ocp}{\textsf{OCP}}

\newcommand{\tocp}{\textsf{tOCP}}
\newcommand{\tocptw}{\textsf{tOCP}\mbox{-}\textsf{tw}\xspace}
\newcommand{\tdm}{\textsf{TDM}}
\newcommand{\tdmtw}{\textsf{TDM}\mbox{-}\textsf{tw}\xspace}

\theoremstyle{definition}

\newtheorem{environment}{Environment}[section]

\newtheorem{lemma}[environment]{Lemma}
\crefname{lemma}{lemma}{lemmata}
\crefformat{lemma}{#2Lemma~#1#3}
\Crefformat{lemma}{#2Lemma~#1#3}

\newtheorem*{lemma*}{Lemma}
\crefname{lemma*}{lemma}{lemmata}
\crefformat{lemma*}{#2Lemma~#1#3}
\Crefformat{lemma*}{#2Lemma~#1#3}

\crefname{proposition}{proposition}{propositions}
\crefformat{proposition}{#2Proposition~#1#3}
\Crefformat{proposition}{#2Proposition~#1#3}

\newtheorem{corollary}[environment]{Corollary}
\crefname{corollary}{corollary}{corollaries}
\crefformat{corollary}{#2Corollary~#1#3}
\Crefformat{corollary}{#2Corollary~#1#3}

\newtheorem{theorem}[environment]{Theorem}
\crefname{theorem}{theorem}{Theorems}
\crefformat{theorem}{#2Theorem~#1#3}
\Crefformat{theorem}{#2Theorem~#1#3}

\newtheorem*{theorem*}{Theorem}
\crefname{theorem*}{theorem}{Theorems}
\crefformat{theorem*}{#2Theorem~#1#3}
\Crefformat{theorem*}{#2Theorem~#1#3}

\newtheorem{conjecture}[environment]{Conjecture}
\crefname{conjecture}{conjecture}{Conjectures}
\crefformat{conjecture}{#2Conjecture~#1#3}
\Crefformat{conjecture}{#2Conjecture~#1#3}

\newtheorem*{hypothesis*}{Hypothesis}
\crefname{hypothesis*}{conjecture}{Conjectures}
\crefformat{hypothesis*}{#2Conjecture~#1#3}
\Crefformat{hypothesis*}{#2Conjecture~#1#3}

\newtheorem{observation}[environment]{Observation}
\crefname{observation}{observation}{Observations}
\crefformat{observation}{#2Observation~#1#3}
\Crefformat{observation}{#2Observation~#1#3}

\crefname{example}{example}{examples}
\crefformat{example}{#2Example~#1#3}
\Crefformat{example}{#2Example~#1#3}

\crefname{remark}{remark}{remarks}
\crefformat{remark}{#2Remark~#1#3}
\Crefformat{remark}{#2Remark~#1#3}

\crefname{figure}{figure}{figures}
\crefformat{figure}{#2Figure~#1#3}
\Crefformat{figure}{#2Figure~#1#3}

\crefname{equation}{equation}{Equations}
\crefformat{equation}{#2#1#3}
\Crefformat{equation}{#2#1#3}

\crefname{chapter}{chapter}{chapters}
\crefformat{chapter}{#2Chapter~#1#3}
\Crefformat{chapter}{#2Chapter~#1#3}

\crefname{section}{section}{sections}
\crefformat{section}{#2Section~#1#3}
\Crefformat{section}{#2Section~#1#3}

\crefname{algorithm}{algorithm}{algorithms}
\crefformat{algorithm}{#2Algorithm~#1#3}
\Crefformat{algorithm}{#2Algorithm~#1#3}

\crefname{notation}{notation}{notations}
\crefformat{notation}{#2Notation~#1#3}
\Crefformat{notation}{#2Notation~#1#3}

\newtheorem{question}[environment]{Question}
\crefname{question}{question}{questions}
\crefformat{question}{#2Question~#1#3}
\Crefformat{question}{#2Question~#1#3}

\crefname{problem}{problem}{problem}
\crefformat{problem}{#2problem~#1#3}
\Crefformat{problem}{#2problem~#1#3}

\crefname{claim}{claim}{claims}
\crefformat{claim}{#2Claim~#1#3}
\Crefformat{claim}{#2Claim~#1#3}

\newtheorem{definition}[environment]{Definition}
\crefname{definition}{definition}{definitions}
\crefformat{definition}{#2Definition~#1#3}
\Crefformat{definition}{#2Definition~#1#3}

\definecolor{CornflowerBlue}{rgb}{0.39, 0.58, 0.93}
\definecolor{Magenta}{rgb}{0.50, 0.0, 0.50}
\definecolor{AppleGreen}{rgb}{0.55, 0.71, 0.0}
\definecolor{AO}{rgb}{0.0, 0.5, 0.0}

\hypersetup{
colorlinks=true,
linkcolor=AO!65!black,
citecolor=AO!65!black,
urlcolor=AppleGreen!65!black,
bookmarksopen=true,
bookmarksnumbered,
bookmarksopenlevel=2,
bookmarksdepth=3
}

\definecolor{DeepCarrotOrange}{rgb}{0.91, 0.41, 0.17}
\definecolor{BananaYellow}{rgb}{1.0, 0.88, 0.21}

\usepackage{graphicx} 

\title{Totally $\Delta$-Modular Tree Decompositions of Graphic Matrices for Integer Programming}
\author[1]{Caleb McFarland\footnote{Supported in part by the National Science Foundation under Grant No. DMS- 2452111.}\footnote{Supported in part by the Georgia Tech ARC-ACO Fellowship.}
}
\affil[1]{School of Mathematics, Georgia Institute of Technology}
\date{June 2025}

\begin{document}

\maketitle

\begin{abstract}
    We introduce the tree-decomposition-based parameter \emph{totally $\Delta$-modular treewidth} (\textsf{TDM}-treewidth) for matrices with two nonzero entries per row. We show how to solve integer programs whose matrices have bounded \textsf{TDM}-treewidth in polynomial time when variables have bounded domain. This extends previous graph-based decomposition parameters for matrices with at most two nonzero entries per row to include matrices with entries outside of $\{-1,0,1\}$. We also give an analogue of the Grid Theorem of Robertson and Seymour for matrices of bounded \textsf{TDM}-treewidth in the language of rooted signed graphs.
\end{abstract}

\section{Introduction}

We are interested in solving integer programs of the form
\[
    \max\{w^Tx : Ax \leq b, x \in [\ell, u] \cap \mathbb{Z}^n\}
\]
where $A \in \mathbb{Z}^{m \times n}, b \in \mathbb{Z}^m, w \in \mathbb{Z}^n$, and $\ell,u \in (\mathbb{Z} \cup \{\infty\})^n$. We may assume that $A$ has at least two nonzero entries per row by possibly changing $\ell,u$. Integer programs generalize many combinatorial optimization problems, see \cite{Schrijver2003Combinatorial}. It is well known that solving integer programs is NP-hard in general, so there has been much research on what conditions on $A$ allow for a polynomial time algorithm. Examples include when $A$ has a constant number of columns (see \cite{Lenstra1983Integer} and improvements in \cite{Kannan1987Minkowskis, Dadush2012integer, reis2023subspace}), a constant number of rows (see \cite{Papdimitriou1981complexity} and improvements in \cite{EisenbrandW2019proximity}), or different block structures \cite{cslovjecsek2021block, cslovjecsek2020efficient, eisenbrand2019algorithmic, brianski2024characterization, cslovjecsek2025parameterized}.

The most classical example is when $A$ is totally unimodular, see \cite{Schrijver2003Combinatorial}. To generalize this, a matrix is called \emph{totally $\Delta$-modular} if every square submatrix of $A$ has determinant bounded by $\Delta$ in absolute value. Hence $\Delta= 1$ is exactly when $A$ is totally unimodular. The following conjecture seeks to generalize the unimodular case to all $\Delta$.

\begin{conjecture}[\cite{shevchenko1996qualitative}]
    For any constant $\Delta$, the integer program $\max\{w^Tx : Ax \leq b, x \in \mathbb{Z}^n\}$ can be solved in polynomial time when $A$ is totally $\Delta$-modular.
\end{conjecture}

Artmann, Weismantel, and Zenklusen~\cite{ArtmannWZ2017Strongly} answered the conjecture in the affirmative for $\Delta = 2$, but $\Delta \geq 3$ remains open. Fiorini, Joret, Weltge, and Yuditsky~\cite{Fiorini2025Integer} gave an algorithm for solving \mis on graphs which do not contain many vertex-disjoint odd cycles. They then used this algorithm to solve totally $\Delta$-modular integer programs with two nonzero entries per row. Kober~\cite{kober2025totally} extends this to totally $\Delta$-modular integer programs with two nonzero entries per row after removing a constant number of rows and columns. Choi, Gorsky, Kim, McFarland, and Wiederrecht~\cite{ChoiGKMW2025OCPtw} extend the work of Fiorini et al. to graphs which ``decompose'' into parts which do not contain many vertex-disjoint odd cycles. They again use this algorithm for \mis to solve the analogous class of integer programs when the coefficients are in $\{-1, 0, 1\}$. Choi et al. also give a grid theorem for when matrices with two nonzero entries per row and entries in $\{-1, 0, 1\}$ decompose into parts which are totally $\Delta$-modular using signed graphs. We extend their work to matrices with entries outside of $\{-1, 0, 1\}$. First, we give a grid theorem which describes the obstructions to decomposing a matrix into totally $\Delta$-modular parts; this theorem is stated in terms of rooted signed graphs. Then we show how to use such a decomposition to solve integer programs.

A \emph{rooted signed graph} is a triple $(G,\gamma,K)$ where $G$ is a graph, $\gamma: E(G) \rightarrow \mathbb{Z}_2$ represents the parity of edges, and $K \subseteq V(G)$ is the set of roots. The odd cycle packing number (\ocp) of a (rooted) signed graph, denoted $\ocp(G,\gamma)$, is the size of a maximum collection of pairwise vertex-disjoint cycles so that the edges of each cycle sum to $1$ under the weighting $\gamma$. Given a matrix $A$ with exactly two nonzero entries per row, we can associate to $A$ a rooted signed graph $G^+_\bullet(A) = (G,\gamma,K)$ where $G$ has edge-vertex incidence matrix with the same support as $A$, $\gamma$ keeps track of which rows have entries with the same or different signs, and $K$ is the set of columns that contain an entry outside of $\{-1, 0, 1\}$. By $G(A)$ we denote the graph $G$, and by $G^+(A)$ we denote the (unrooted) signed graph $(G,\gamma)$. See \cref{sec:Preliminaries} for precise definitions.

Informally, a matrix with two nonzero entries per row has bounded \emph{totally $\Delta$-modular treewidth} if it has a tree decomposition such that each bag is totally $\Delta$-modular. In a rooted signed graph, this corresponds to each bag having few elements in $K$ and bounded odd cycle packing number. However, the precise definition has technical restrictions on the adhesions. For each bag we have an apex set, the ``protector'', and adhesions outside this apex set all have size 0 or 1 depending on whether the protector is ``strong''. This requirement is to allow for dynamic programming: If a bag has an unbounded number of children and each contributes new vertices in their adhesion, this can pose a problem for dynamic programming. The reason that some protectors need not be strong is because we may eventually assume all variables of the integer program are in $\{0,1\}$, see \cref{sec:futherDiscussion} for more details. We now formally define the totally $\Delta$-modular treewidth (\tdm-treewidth) of a rooted signed graph (and thus a matrix with two nonzero entries per row).

\begin{definition}[\tdm-treewidth]
    Let $(G,\gamma,K)$ be a rooted signed graph. Given a tree decomposition $(T,\beta)$ of $(G,\gamma,K)$, a \emph{protector} for a bag $\beta(t)$ is a set $\alpha(t) \subseteq \beta(t)$ that contains all the roots in the bag and such that all adhesions to $t$ have size at most one after deleting $\alpha(t)$. That is, $K \cap \beta(t) \subseteq \alpha(t)$ and $|(\beta(t) \cap \beta(t')) \setminus \alpha(t)| \leq 1$ for all $tt' \in E(T)$. A protector is \emph{strong} if all the adhesions to $t$ are left empty after deleting $\alpha(t)$, i.e. if $\beta(t) \cap \beta(t') \subseteq \alpha(t)$ for all $tt' \in E(T)$.

    A \emph{\tdm-tree-decomposition} of a rooted signed graph $(G, \gamma, K)$ is a tuple $(T,\beta,\alpha,J)$ where $T$ is a tree, $J$ is a (possibly empty) subtree of $T$, and $\beta,\alpha: V(T) \rightarrow 2^{V(G)}$ all satisfy
    \begin{enumerate}
        \item $(T,\beta)$ is a tree-decomposition of $G$,
        \item $\alpha(t)$ is a protector for $\beta(t)$ for all $t \in V(T)$ and a strong protector when $t \in V(J)$,
        \item $K \subseteq \bigcup_{j \in V(J)} \beta(j)$,
    \end{enumerate}
    The \emph{width} of a \tdm-tree-decomposition is
    \[\max_{t \in V(T)} |\alpha(t)| + \ocp(G[\beta(t) \setminus \alpha(t)], \gamma).\]
    The \emph{\tdm-treewidth} of $(G,\gamma,K)$, denoted $\tdmtw(G,\gamma,K)$, is the minimum width of a $\tdm$-tree-decomposition.
\end{definition}

The definition for $\tdm$-treewidth arises naturally out of other width parameters for rooted graphs and signed graphs, see \cref{sec:PropertiesTDMTW}. See \cref{sec:futherDiscussion} for possible modifications of the definition and further discussion.

We are now ready to state our main algorithmic result, where we assume that $w,b,\ell$, and $u$ can be represented in $\mathsf{poly}(n)$ many bits.

\begin{restatable}{theorem}{MainAlgorithmicTheorem}\label{thm:MainAlgorithm}
    Let $d,k$ be nonnegative integers, and let $A$ be a matrix with two nonzero entries per row whose associated rooted signed graph has $\tdm$-treewidth at most $k$. Then for any $w \in \mathbb{Z}^n, b \in \mathbb{Z}^m$, and $\ell,u \in \mathbb{Z}^n$ with $\|u - \ell\|_\infty \leq d$, we can solve the integer program
    \[\max\{w^Tx : Ax \leq b, x \in [\ell,u] \cap  \mathbb{Z}^n\}\]
    in $d^{\mathcal{O}(k)}n^{f(k)}$ time for some computable $f$.
\end{restatable}

We prove \cref{thm:MainAlgorithm} in \cref{sec:SolvingIP} via dynamic programming and a reduction to \mis. We note that in the above theorem we do not require a bound on $\|A\|_\infty$, the largest entry of $A$ in absolute value. We do however require a bound on $\|u - \ell\|_\infty$, which is natural for many combinatorial problems. This need arises out of the lack of a proximity result like the one of Cook et al.~\cite[Theorem~1]{cook1986sensitivity} (see also \cite{paat2020distances, celaya2022improving}) for matrices which are only ``locally" totally $\Delta$-modular. In particular, the algorithm is still poly-time for $d \in \mathsf{poly}(n)$, which is in line with the bound of $d = n\Delta$ we would get for totally $\Delta$-modular matrices by the previously mentioned proximity result of Cook et al.

We also give a grid theorem for having bounded $\tdm$-treewidth. We first define the relevant grids. The \emph{cylindrical $(n \times m)$-grid} is $P_n \Box C_m$, the Cartesian (or box) product of the path on $n$ vertices $P_n$ and the cycle $C_m$. Let $\gamma_0, \gamma_1$ be the functions which are identically equal to 0 or 1 respectively.

\begin{definition}[Parity Handle $\mathcal{H}_k$ and parity vortex $\mathcal{V}_k$]
    Let $W$ be the cylindrical $(k \times 4k)$-grid which we embed naturally in the plane such that the cycles $C_{4k}$ are concentric. Let $X = \{x_1, \dots, x_{2k}\}$ consist of every other vertex along the outer face in clockwise order.
    
    The \emph{unsigned, unrooted parity handle} $\mathbf{H}_k$ of order $k$ is the graph obtained from $W$ by adding the edges $x_ix_{2k-i+1}$ for all $i \in [k]$. The \emph{parity handle} $\mathcal{H}_k$ of order $k$ is the rooted signed graph $(\mathbf{H}_k,\gamma_1, \varnothing)$.
    
    The \emph{unsigned, unrooted parity vortex} $\mathbf{V}_k$ of order $k$ is the graph obtained from $W$ by adding the edges $x_{2i-1}x_{2i}$ for all $i \in [k]$. The \emph{parity vortex} $\mathcal{V}_k$ of order $k$ is the rooted signed graph $(\mathbf{V}_k,\gamma_1, \varnothing)$.

    See \cref{fig:ParityGrids} for examples.
\end{definition}

\begin{figure}
    \centering
    \begin{subfigure}[b]{0.75\textwidth}
    \centering
    \resizebox{!}{5.5cm}{
        \begin{tikzpicture}[
            mystyle/.style={circle,draw,fill=black,minimum size=5}
        ]
        
        \pgfmathsetmacro{\w}{4}
        \pgfmathsetmacro{\l}{19}
        \pgfmathsetmacro{\lp}{\l + 1}
        \pgfmathsetmacro{\wm}{\w - 1}
        
        \node (o1) at ($(0,0) + (0:12)$) {};
        \node (o2) at ($(0,0) + (180:12)$) {};
        
        \foreach \y in {0,...,\w} {
            \pgfmathsetmacro{\r}{\y + 6}
            \draw[fill=none] (o1) circle (\r);
            \draw[fill=none] (o2) circle (\r);
        }
        
        \foreach \x in {0,...,\l} {
            \pgfmathsetmacro{\angle}{\x * 360 / \lp}
            \foreach \y in {0,...,\w} {
                \pgfmathsetmacro{\r}{\y + 6}
                \node[mystyle] (\x_\y_1) at ($(o1) + (\angle:\r)$) {};
                \node[mystyle] (\x_\y_2) at ($(o2) + (\angle:\r)$) {};
            }
        }
        
        \foreach \x in {0,...,\l} {
            \foreach \y [count=\yi] in {0,...,\wm} {
                \draw[color=black] (\x_\y_1) -- (\x_\yi_1);
                \draw[color=black] (\x_\y_2) -- (\x_\yi_2);
            }
        }
        
        \foreach \x in {0,...,\w} {
            \pgfmathtruncatemacro{\everyotherx}{2*\x+1}
            \pgfmathtruncatemacro{\oppx}{\lp-\everyotherx}
            \pgfmathtruncatemacro{\angle}{\w-\everyotherx + 1}
            \pgfmathtruncatemacro{\outangle}{-90 - 10*\angle}
            \pgfmathtruncatemacro{\inangle}{90 + 10*\angle}
            \draw[red, line width=1mm] (\everyotherx_0_2) to[in=\inangle, out=\outangle] (\oppx_0_2);
        }
        \foreach \x in {0, ..., \w} {
            \pgfmathtruncatemacro{\firstx}{4*\x}
            \pgfmathtruncatemacro{\nextx}{4*\x + 2}
            \pgfmathsetmacro{\firstangle}{\firstx * 360 / \lp}
            \pgfmathsetmacro{\nextangle}{\nextx * 360 / \lp}
            \pgfmathsetmacro{\inang}{\nextangle - 150}
            \pgfmathsetmacro{\outang}{\firstangle + 150}
            \draw[red, line width=1mm] (\firstx_0_1) to[in=\inang, out=\outang] (\nextx_0_1);
        }
        \end{tikzpicture}
    }
    \caption{The parity handle of order 5 (left) and the parity vortex of order 5 (right). Both graphs are shifting equivalent to the graph with odd edges exactly equal to the red edges.}
    \label{fig:ParityGrids}
    \end{subfigure}
    \\[5mm]
    \begin{subfigure}[b]{0.75\textwidth}
    \centering
    \resizebox{!}{5.5cm}{
        \begin{tikzpicture}[
            mystyle/.style={circle,draw,fill=black,minimum size=5}
        ]
        
        \pgfmathsetmacro{\w}{8}
        \pgfmathtruncatemacro{\wm}{\w - 1}
        
        \foreach \x in {1,...,\w} {
            \foreach \y in {1,...,\wm} {
                \node[mystyle] (v\x\y) at (2*\x,2*\y) {};
            }
        }
        \foreach \x in {1,...,\w} {
            \foreach \y in {\w} {
                \node[mystyle, fill=red, red] (v\x\y) at (2*\x,2*\y) {};
            }
        }
        
        \foreach \x in {1,...,\wm} {
            \foreach \y in {1,...,\w} {
                \draw[line width=0.8mm] (v\x\y) -- (v\the\numexpr\x+1\relax\y);
            }
        }
        
        \foreach \x in {1,...,\w} {
            \foreach \y in {1,...,\wm} {
                \draw[line width=0.8mm] (v\x\y) -- (v\x\the\numexpr\y+1\relax);
            }
        }
        \end{tikzpicture}
    }
    \caption{The rooted grid of order 8. The vertices in red are the rooted vertices $K_{\mathbf{W}_k}$. All cycles of the rooted grid are even.}
    \label{fig:RootedGrid}
    \end{subfigure}
\end{figure}

\begin{definition}[Rooted Grid $\mathcal{W}_k$]\label{def:rootedGrid}
    Let $\mathbf{W}_k$ be the $(k \times k)$-grid and let $K_{\mathbf{W}_k}$ be the vertex set of the first row of $\mathbf{W}_k$. The rooted grid $\mathcal{W}_k$ of order $k$ is the rooted signed graph $(\mathbf{W}_k, \gamma_0, K_{\mathbf{W}_k})$. See \cref{fig:RootedGrid} for an example.
\end{definition}

        
        
        
        

We can now state our main structural result.

\begin{restatable}{theorem}{MainGridTheorem}\label{thm:MainGridTheorem}
    There exists a polynomial $f: \mathbb{N} \rightarrow \mathbb{N}$ such that for every integer $k \geq 1$ and every rooted signed graph $(G,\gamma,K)$,
    \begin{enumerate}
        \item if $(G,\gamma,K)$ contains one of $\mathcal{H}_k, \mathcal{V}_k$, or $\mathcal{W}_k$ as a minor, then $\tdmtw(G,\gamma,K) \geq \Omega(k)$, and
        \item if $\tdmtw(G,\gamma,K) \geq f(k)$, then $G$ contains one of $\mathcal{H}_k,\mathcal{V}_k,$ or $\mathcal{W}_k$ as a minor.
    \end{enumerate}
\end{restatable}

\cref{thm:MainGridTheorem} and \cref{thm:MainAlgorithm} give the following immediate corollary.

\begin{corollary}\label{cor:MainAlgorithmGridVersion}
    Let $d,k$ be nonnegative integers, and let $A$ be a matrix with two nonzero entries per row whose associated rooted signed graph forbids $\mathcal{H}_k, \mathcal{V}_k$, and $\mathcal{W}_k$ as minors. Then for all $w \in \mathbb{Z}^n$, $b \in \mathbb{Z}^m$, and $\ell,u \in \mathbb{Z}^n$ with $\|u - \ell\|_\infty \leq d$, we can solve the integer program
    \[\max\{w^Tx : Ax \leq b, x \in [\ell,u]\cap \mathbb{Z}^n\}\]
    in $d^{\mathsf{poly}(k)}n^{f(k)}$ time for some computable $f$.
\end{corollary}

Our work is towards the following problem posed in \cite{ChoiGKMW2025OCPtw}.

\begin{question}[Question 10.10~\cite{ChoiGKMW2025OCPtw}]\label{ques:LargestClassOfMatrices}
    What is the most general class of matrices with two nonzero entries per row such that we can solve the integer program $\max\{w^T x : Ax \leq b, x \in [\ell,u] \cap \mathbb{Z}^n\}$ in polynomial time for all matrices A in the class?
\end{question}

This question seeks to generalize the below question which is central to algorithmic graph theory.

\begin{question}\label{ques:MIS}
    What is the most general class of graphs for which we can solve \mis (\textsc{MIS}) in polynomial time for all graphs in the class?
\end{question}

In order to make the above questions tractable, it is natural to only consider classes which are closed under certain operations. For example, if we ask \cref{ques:MIS} for hereditary graph classes, then there has been a long line of research on solving \textsc{MIS} over graph classes defined by forbidden induced subgraphs \cite{grotschel1984polynomial, minty1980maximal, alekseev2004polynomial, lozin2008polynomial, corneil1981complement, farber1993upper}.

More relevant to this work is when we force the classes to be closed under the minor relation. Then \cref{ques:MIS} is answered by the famous Grid Theorem of Robertson and Seymour~\cite{RobertsonS1986Graph5}, which implies that \textsc{MIS} can be solved on a minor-closed class of graphs if and only if the class excludes a planar graph. If we restrict our matrices in \cref{ques:LargestClassOfMatrices} to be in $\{0, 1\}^{m \times n}$ and use the typical minor relation on $G(A)$, then \cref{ques:LargestClassOfMatrices} is equivalent to \cref{ques:MIS}.

If we instead restrict our matrices in \cref{ques:LargestClassOfMatrices} to be over $\{-1,0,1\}$, it is natural to associate such matrices with signed graphs and to force classes to be closed under the signed graph minor relation (see \cref{sec:signedGraphs} for relevant definitions). Historically, signed graphs have been studied in relation to matroids and geometry \cite{zaslavsky1982signed, zaslavsky1991orientation, zaslavsky2013signed}, but they have recently found relations to integer programming \cite{ChoiGKMW2025OCPtw, kober2025totally}. \cref{ques:LargestClassOfMatrices} for signed graph minor closed classes of matrices with entries in $\{-1,0,1\}$ then again reduces to \cref{ques:MIS} for graph classes closed under odd-minors, see \cite[Theorem 10.3]{ChoiGKMW2025OCPtw} based on \cite[Section 3]{Fiorini2025Integer}. We discuss the relation between odd-minors and signed graphs in \cref{sec:oddMinorsAndSignedGraphs}. Choi et al.~\cite{ChoiGKMW2025OCPtw} showed that \textsc{MIS} (and hence the analogous integer programming problem) is solvable over classes which exclude both a planar graph with at most two odd faces and a planar graph with all odd cycles touching a common face (that is, $\mathbf{H}_k$ and $\mathbf{V}_k$ for some $k$) as an odd-minor. This generalized an array of previous work on the \textsc{MIS} problem in odd-minor-closed graph classes \cite{eiben2021measuring, jansen2021vertex, jansen2021fpt, gollin2023structure, jaffke2025dynamic, Fiorini2025Integer}. Other notable work on the \textsc{MIS} in odd-minor-closed graph classes include \cite{gerards1989min, tazari2012faster}. The problem is only known to be hard when the class contains all planar graphs, and so an exact hardness threshold is not known for this case.

We instead pose \cref{ques:LargestClassOfMatrices} for matrices which may have entries outside of $\{-1, 0, 1\}$. We again force our classes to be minor closed, but it is not obvious which minor operation to choose. A matrix with two nonzero entries per row is totally $\Delta$-modular if and only if it has bounded entries, a bounded number of columns with entries outside $\{-1,0,1\}$, and the associated signed graph has bounded odd cycle packing (see \cref{lem:deltamodularequivalence}). Motivated by this equivalence, we propose representing such matrices as rooted signed graphs as defined above. We can then naturally encode the minor relation as \emph{rooted signed graph minors} (see \cref{sec:rootedSignedGraphs} for relevant definitions), which allows us to apply results on rooted minors and rooted tree decompositions~\cite{MarxSW2017Rooted, JansenS2024steiner, HodorLMR2024quickly}.

The best known results for matrices with two nonzero entries per row and entries possibly outside $\{-1,0,1\}$ are given by the following results which more generally apply to matrices with an arbitrary number of nonzero entries per row. The \emph{primal graph} of a matrix $A$ is a graph with vertex set equal to the columns of $A$ where two columns are adjacent if and only if there exists a row where both columns are nonzero. If we assume the domain size $\|u - \ell\|_\infty$
is bounded, then we can solve the integer program when the primal graph has bounded treewidth \cite{Jansen2015structural,Freuder1990Complexity}. If we instead bound $\|A\|_\infty$, then we can solve the integer program when the corresponding graph has bounded treedepth\footnote{A graph has bounded \emph{treedepth} if it has a tree decomposition $(T,\beta)$ of bounded width such that $T$ has bounded height.} \cite{eisenbrand2019algorithmic, GanianOR2017goingbeyond, ganian2018complexity}. Integer programming remains NP-hard on matrices with $\|A\|_\infty$ bounded and where the primal graph has bounded treewidth \cite{GanianOR2017goingbeyond}, so treedepth can not be relaxed to treewidth. Similarly, integer programming remains NP-hard on matrices where the primal graph has bounded treedepth and $\|A\|_\infty$ is unbounded \cite{Dvorak2021complexity}. If we consider the matroid represented by $A$, integer programming remains NP-hard even when the matroid has bounded branch-width and the entries are in $\{-1,0,1\}$ \cite{cunningham2007integer}. However, all of these reductions \cite{GanianOR2017goingbeyond, Dvorak2021complexity, cunningham2007integer} (as well as others \cite{knop2020tight, eiben2019integer, brianski2025integer}) rely on matrices with at least 3 nonzero entries per row. To the best of the author's knowledge, the only known hardness result for matrices with at most two nonzero entries per row is the reduction to \textsc{Maximum (Weighted) Independent Set}.

The above results imply that if $\|u - \ell\|_\infty$ is bounded and if $A$ is a matrix with two nonzero entries per row such that $G(A)$ has bounded treewidth, then we can solve the integer program in polynomial time. \cref{thm:MainAlgorithm} generalizes both this result and the result of Choi et al.~\cite{ChoiGKMW2025OCPtw} (see \cref{thm:IPsolutionForForbiddingParityGrids}) when $\|u - \ell\|_\infty$ is bounded.

The paper is organized as follows. In \cref{sec:Preliminaries}, we give the relevant background on signed graphs, rooted graphs, and their connection to totally $\Delta$-modular matrices. In \cref{sec:Kfreetw}, we discuss $K$-free treewidth, as introduced in \cite{JansenS2024steiner}, and a grid theorem for rooted minors. In \cref{sec:ocptw}, we discuss \ocp-treewidth, as introduced in \cite{ChoiGKMW2025OCPtw}, and its applications to integer programming. In \cref{sec:PropertiesTDMTW}, we discuss the connection between $K$-free treewidth, \ocp-treewidth, and \tdm-treewidth. We use this connection to give an FPT-approximation algorithm for the \tdm-treewidth of a graph, and we prove \cref{thm:MainGridTheorem}. In \cref{sec:SolvingIP}, we prove \cref{thm:MainAlgorithm}. Finally, in \cref{sec:futherDiscussion}, we discuss potential future work and open questions.

\section{Preliminaries}\label{sec:Preliminaries}
We use $[k]$ to denote the set $\{1, 2, \dots, k\}$. By $\|A\|_\infty$ we denote the largest absolute value among all entries of the matrix $A$. All graphs may have parallel edges but not loops. We denote by $G[X]$ the graph induced on $X$ for $X \subseteq V(G)$. For $V_1, V_2$ a partition of $V(G)$, we denote by $[V_1,V_2]$ the set of all edges with one end in $V_1$ and one end in $V_2$. We call such a $[V_1,V_2]$ an \emph{edge cut}.


\begin{definition}[tree decomposition]
    A \emph{tree decomposition} of a graph $G$ is a pair $(T, \beta)$ where $T$ is a tree and $\beta: V(T) \rightarrow 2^{V(G)}$ such that
    \begin{enumerate}
        \item $\bigcup_{t \in V(T)} \beta(t) = V(G)$,
        \item for every edge $e \in E(G)$, there exists $t \in V(T)$ such that both ends of $e$ are contained in $\beta(t)$, and
        \item for every $v \in V(G)$, the subgraph of $T$ induced by $\{t \in V(T) : v \in \beta(t)\}$ is connected.
    \end{enumerate}
    The \emph{width} of a tree decomposition $(T, \beta)$ is $\max_{t \in V(T)} |\beta(t)| - 1$. The \emph{treewidth} of a graph $G$ is the minimum width of a tree decomposition of $G$. We call $\beta(t)$ the \emph{bag} of $t$. Each edge $tt' \in E(T)$ corresponds to an \emph{adhesion} $\beta(t) \cap \beta(t')$.
\end{definition}

\begin{definition}[grid]\label{def:grid}
    The $k \times k$ grid is the graph with vertex set $[k]^2$ such that $(i,j)$ is adjacent to $(i',j')$ if and only if $|i - i'| + |j - j'| = 1$.
\end{definition}

\subsection{Signed graphs}\label{sec:signedGraphs}
A signed graph is a pair $(G, \gamma)$ where $G$ is a graph and $\gamma: E(G) \rightarrow \mathbb{Z}_2$. For $e \in E(G)$, if $\gamma(e) = 0$ the edge is said to be \emph{even}, and if $\gamma(e) = 1$ the edge is said to be \emph{odd}. The \emph{parity} of a path or cycle $C$ in $G$ is $\sum_{e \in E(C)} \gamma(e)$, and a cycle is said to be \emph{even} or \emph{odd} if the parity is 0 or 1 respectively. The \emph{odd cycle packing number (\ocp)} of a signed graph $(G,\gamma)$, denoted $\ocp(G,\gamma)$ is the size of a maximum collection of pairwise vertex-disjoint odd cycles with respect to $\gamma$.

We commonly use $\gamma_1$ to denote the function which maps every edge to $1 \in \mathbb{Z}_2$ where the domain of the function is clear from context. Similarly we use $\gamma_0$ to denote the function which maps everything to $0 \in \mathbb{Z}_2$.

Given a signed graph $(G, \gamma)$, \emph{shifting at a vertex} $v \in V(G)$ replaces $\gamma(e)$ with $\gamma(e) + 1$ for each edge $e$ incident to $v$. We say $(G,\gamma')$ is obtained from $(G,\gamma)$ by \emph{shifting} if it is obtained by a sequence of shiftings at vertices. Note that if $(G,\gamma')$ is a shifting of $(G,\gamma)$, then every cycle in $(G,\gamma')$ has the same parity as in $(G,\gamma)$. The converse is also true: if $(G,\gamma')$ and $(G,\gamma)$ have the same parity for every cycle, then $(G,\gamma')$ can be obtained from $(G,\gamma)$ by shifting. Hence it is often natural to think of signed graphs under the equivalence relation given by shifting.

A signed graph $(H, \gamma_H)$ is a \emph{(signed graph) minor} of $(G, \gamma_G)$ if it can be obtained by vertex deletion, edge deletion, shifting, and contracting even edges. Because we may only contract even edges, every cycle in $(H,\gamma_H)$ naturally corresponds to a cycle in $(G,\gamma_G)$ of the same parity. Equivalently, $(H,\gamma_H)$ is a minor of $(G,\gamma_G)$ if there exists a \emph{minor model} of $(H,\gamma_H)$ in $(G,\gamma_G)$. That is, if there exist
\begin{enumerate}
    \item a mapping $\varphi_V$ from $V(H)$ to pairwise vertex disjoint trees in $G$,
    \item an injection $\varphi_E:E(H) \rightarrow E(G)$ such that $\varphi_E(uv)$ has one end in $\varphi_V(u)$ and one end in $\varphi_V(v)$ for each $uv \in E(H)$, and
    \item a shifting $(G,\gamma_G')$ of $(G,\gamma_G)$ such that
    \begin{enumerate}
        \item for all $v \in V(H)$, the edges of the tree $\varphi_V(v)$ all have label 0, and
        \item for every $e \in E(H)$, $\gamma_G'(\varphi_E(e)) = \gamma_H(e)$.
    \end{enumerate}
\end{enumerate}
Note that if $(H,\gamma_H)$ is a signed graph minor of $(G,\gamma_G)$ then $H$ is a minor of $G$ in the usual sense.

A signed graph $(G,\gamma_G)$ contains a signed graph $(H, \gamma_H)$ as a \emph{subdivision} if there exists $\gamma_H', \varphi_V, \varphi_E$ such that
\begin{enumerate}
    \item $(H,\gamma'_H)$ is a shifting of $(H,\gamma_H)$,
    \item $\varphi_V: V(H) \rightarrow V(G)$ is an injection, and
    \item $\varphi_E$ is a mapping from edges of $H$ to internally vertex-disjoint paths in $G$ such that
    \begin{enumerate}
        \item $\varphi_E(uv)$ is a path with ends $\varphi_V(u),\varphi_V(v)$ and
        \item the parity of the path $\varphi_E(uv)$ in $(G,\gamma)$ is equal to $\gamma'_H(uv)$.
    \end{enumerate}
\end{enumerate}
Note that if $(G,\gamma_G)$ contains $(H,\gamma_H)$ as a subdivision, then $(G,\gamma_G)$ contains $(H,\gamma_H)$ as a minor.

We call the cycles of length 4 in a grid the \emph{cells}. A signed $n \times m$ grid is said to be \emph{even} if all of its cells are even. Note that if $(W,\gamma)$ is an even grid, then $(W,\gamma_0)$ is a shifting. We now show that in any signed grid, we can always find a large even grid.

\begin{figure}[h]
    \centering
    \scalebox{0.4}{
        \begin{tikzpicture}[mystyle/.style={circle,draw,fill=black}]
            \node[mystyle, label={[font=\huge]120:$(f(1),f(1))$}] (a1_b1) at (1,-1) {};
            \node[mystyle, label={[font=\huge]90:$(f(1),f(2))$}] (a1_b2) at (12,-1) {};
            \node[mystyle, label={[font=\huge]90:$(f(1),f(3))$}] (a1_b3) at (23,-1) {};
            \node[mystyle, label={[font=\huge]180:$(f(2),f(1))$}] (a2_b1) at (1,-12) {};
            \node[mystyle] (a2_b2) at (12,-12) {};
            \node[mystyle] (a2_b3) at (23,-12) {};
            \node[mystyle, label={[font=\huge]180:$(f(3),f(1))$}] (a3_b1) at (1,-23) {};
            \node[mystyle] (a3_b2) at (12,-23) {};
            \node[mystyle] (a3_b3) at (23,-23) {};
    
            \draw[line width = 0.8mm] (a1_b1) -- (a1_b3) -- (25,-1);
            \draw[line width = 0.8mm] (a2_b1) -- (a2_b3) -- (25,-12);
            \draw[line width = 0.8mm] (a3_b1) -- (a3_b3) -- (25,-23);
            \draw[line width = 0.8mm] (a1_b1) -- (a3_b1) -- (1,-25);

            \draw[blue, line width = 0.8mm] (2,-2) -- (11,-2) -- (11,-11) -- (2,-11) -- (2,-2);
            \draw[blue, line width = 0.8mm] (13,-2) -- (22,-2) -- (22,-11) -- (13,-11) -- (13,-2);
            \draw[blue, line width = 0.8mm] (2,-13) -- (11,-13) -- (11,-22) -- (2,-22) -- (2,-13);
            \draw[blue, line width = 0.8mm] (13,-13) -- (22,-13) -- (22,-22) -- (13,-22) -- (13,-13);

            \node[blue] (H12) at (6.5, -6.5) {\Huge$H_{1,2}$};
            \node[blue] (H13) at (17.5, -6.5) {\Huge$H_{1,3}$};
            \node[blue] (H22) at (6.5, -17.5) {\Huge$H_{2,2}$};

            \draw[line width = 0.8mm] (a1_b2) -- (12,-2) -- (11,-2);
            \draw[line width = 0.8mm] (a1_b3) -- (23,-2) -- (22,-2);
            \draw[line width = 0.8mm] (a2_b2) -- (12,-13) -- (11,-13);
            \draw[line width = 0.8mm] (a2_b3) -- (23,-13) -- (22,-13);

            \draw[line width = 0.8mm] (a2_b2) -- (12,-11) -- (11,-11);
            \draw[line width = 0.8mm] (a2_b3) -- (23,-11) -- (22,-11);
            \draw[line width = 0.8mm] (a3_b2) -- (12,-22) -- (11,-22);
            \draw[line width = 0.8mm] (a3_b3) -- (23,-22) -- (22,-22);

            \draw[red, line width = 0.8mm] (17,-16) -- (18, -16) -- (18,-17) -- (17,-17) -- (17,-16);
            \node[red] (odd) at (15.5, -16.5) {\huge odd cell};
            \draw[line width = 0.8mm] (22,-22) -- (18,-22) -- (18,-17);
            \draw[line width = 0.8mm] (22,-13) -- (18,-13) -- (18,-16);

            \node[scale=2] (d1) at (27,-6.5) {\Huge \dots};
            \node[scale=2] (d2) at (27,-17.5) {\Huge \dots};
            \node[scale=2] (d3) at (6.5, -24) {\Huge \vdots};
        \end{tikzpicture}
    }
    \caption{The construction in the proof of \cref{lem:totallyevengrid}.}
\end{figure}

\begin{lemma}\label{lem:totallyevengrid}
    Let $k \in \mathbb{N}$. Let $G$ be the $k^2 \times k^2$ grid and let $W$ be the $k \times k$ grid. Then $(G, \gamma)$ contains $(W,\gamma_0)$ as a subdivision for every $\gamma$.
\end{lemma}
\begin{proof}
    Enumerate the vertices of $G$ as $V(G) = \{(i,j) : i,j \in [k^2]\}$ as in \cref{def:grid}. Similarly let the vertices of $W$ be enumerated $(a,b)$ for $a,b \in [k]$. Let $f(x) = (k+1)(x-1) + 1$. We will map the vertex $(a,b)$ in $W$ to the vertex $(f(a), f(b))$ in $G$. For every $b \in [k-1]$, the edge between $(a,b)$ and $(a,b+1)$ will be mapped to the path in $G$ with vertex set $\{f(a), f(b) + \ell) : 0 \leq \ell \leq k+1\}$. For every $a \in [k-1]$, the edge between $(a,1)$ and $(a+1,1)$ will be mapped to the path with vertex set $\{f(a) + \ell, 1) : 0 \leq \ell \leq k+1\}$. It remains to map the edges between $(a,b)$ and $(a+1,b)$ for $b \geq 2$.

    For $a \in \{2,3, \dots, k\}$ and $b \in [k-1]$, let $H_{a,b}$ be the $k \times k$ subgrid of $G$ with opposite corners at $(f(a) + 1, f(b-1) + 1)$ and $(f(a+1) - 1, f(b) - 1)$. We may assume that each $H_{a,b}$ contains an odd cell, otherwise we take $W = H_{a,b}$. We claim that between any two corners of the grid $H_{a,b}$, there exists a path in $H_{a,b}$ of both parities. Indeed, we make take two vertex-disjoint paths from the two corners to the vertices of an odd cell. Then by deciding which way to traverse through the odd cell we may decide the parity of the resulting path.

    We fix $a \in [k-1]$ and define the mapping from the edge between $(a,b)$ and $(a,b+1)$ in order of increasing $b$. The first two vertices of the path are $(f(a) + 1, f(b))$ and $(f(a)+1, f(b)-1)$. The last two vertices of the path are $(f(a+1)-1, f(b)-1)$ and $(f(a+1) - 1, f(b))$. Note that $(f(a) + 1, f(b) - 1)$ and $(f(a+1)-1, f(b)-1)$ are two corners of the grid $H_{a,b}$. Thus we may complete the path through $H_{a,b}$ with a parity of our choosing. We choose the path that forces the cell with vertices $(a,b-1), (a,b), (a+1,b), (a+1,b-1)$ to get mapped to an even cycle. Because the mapping from all other edges of the cell are fixed when choosing the edge between $(a,b)$ and $(a,b+1)$, such a choice is always possible. By choosing the paths in this way, every cell is mapped to an even cycle.
\end{proof}

The above lemma implies that $(G,\gamma)$ forbids an even grid as a signed graph minor if and only if $G$ forbids a grid as a minor. More precisely, we get the following corollary.

\begin{corollary}\label{cor:forbidGridForbidEvenGridEquivalence}
    For any $\gamma$, the following holds.
    If $G$ forbids the $k \times k$ grid as a minor, then $(G,\gamma)$ forbids the $k \times k$ even grid as a signed graph minor.
    Conversely, if $(G, \gamma)$ forbids the $k \times k$ even grid as a signed graph minor, then $G$ forbids the $k^2 \times k^2$ grid as a minor.
\end{corollary}

\subsection{Odd-minors and signed graphs}\label{sec:oddMinorsAndSignedGraphs}
A graph $G$ contains a graph $H$ as an \emph{odd-minor} if $H$ can be obtained from $G$ by deleting vertices, deleting edges, and contracting edge cuts. Equivalently, $H$ is an odd-minor of $G$ if there exists
\begin{enumerate}
    \item a mapping $\varphi_V$ from $V(H)$ to pairwise vertex disjoint trees in $G$,
    \item an injection $\varphi_E:E(H) \rightarrow E(G)$ such that $\varphi_E(uv)$ has one end in $\varphi_V(u)$ and one end in $\varphi_V(v)$ for each $uv \in E(H)$, and
    \item a 2-coloring $c:V(G) \rightarrow [2]$ such that
    \begin{enumerate}
        \item for all $v \in V(H)$, $c$ is a proper 2-coloring when restricted to the tree $\varphi_V(v)$, and
        \item for every $e \in E(H)$, $\varphi_E(e)$ is monochromatic.
    \end{enumerate}
\end{enumerate}

Odd-minors have been studied extensively both as it relates to coloring \cite{geelen2009odd, steiner2022asymptotic, kuhn2025disproof} and algorithms \cite{gerards1989min, tazari2012faster, gollin2023structure, jaffke2025dynamic, ChoiGKMW2025OCPtw}. Odd-minors and signed graph minors are essentially equivalent due to the following observations. Recall that $\gamma_1$ is the function which maps every edge to $1 \in \mathbb{Z}_2$.

\begin{observation}\label{obs:oddMinorSignedGraphTotallyOdd}
    $G$ contains $H$ as an odd-minor if and only if $(G,\gamma_1)$ contains $(H,\gamma_1)$ as a signed graph minor.
\end{observation}
\begin{proof}
    Suppose $G$ contains $H$ as an odd-minor. Consider the sequence of deletions and contractions. Clearly whenever we delete a vertex or edge in $G$ we may do the same in $(G,\gamma_1)$. Whenever we contract an edge cut $[V_1,V_2] \subseteq E(G)$, we can shift in $(G,\gamma_1)$ at every vertex in $V_1$ such that the even edges are exactly the edges $[V_1,V_2]$. We may then contract the edges in the signed graph.

    Suppose $(G,\gamma_1)$ contains $(H,\gamma_1)$ as a signed graph minor. Consider the minor model $\varphi_V, \varphi_E$ of $(H,\gamma_1)$ in $(G,\gamma_1)$. Let $c:V(G) \rightarrow [2]$ denote whether a vertex was shifted at in $G$ in order to make the edges of $\varphi_V(v)$ have label 0 for every $v \in V(H)$. Then on every tree $\varphi_V(v)$, $c$ is a proper 2-coloring. Because $(G, \gamma_1)$ and $(H,\gamma_1)$ both have all edges odd, every edge in the image of $\varphi_E$ must be monochromatic under $c$. Thus we obtain a minor model for $H$ in $G$ as an odd-minor.
\end{proof}

\begin{observation}\label{obs:oddMinorSignedGraphEquivalenceMinDeg3}
    Let $H$ be a graph of minimum degree 3. Let $(G,\gamma)$ be a signed graph, and let $G'$ be the graph formed from $G$ by subdividing every even edge exactly once. Then $(G,\gamma)$ contains $(H,\gamma_1)$ as a signed graph minor if and only if $G'$ contains $H$ as an odd-minor.
\end{observation}
\begin{proof}
    Suppose $(G,\gamma)$ contains $(H,\gamma_1)$ as a minor. Clearly $(G',\gamma_1)$ contains $(G,\gamma)$ as a minor by shifting at every subdividing vertex and contracting one of the incident edges. Thus $(G',\gamma_1)$ contains $(H,\gamma_1)$ as a minor, and so $G'$ contains $H$ as an odd-minor by \cref{obs:oddMinorSignedGraphTotallyOdd}.

    Suppose $G'$ contains $H$ as an odd-minor. Then $(G',\gamma_1)$ contains $(H,\gamma_1)$ as a minor. Because $H$ has minimum degree 3, every subdividing vertex must be contracted or deleted when forming $H$. We may reorder the operations such that all subdividing vertices are removed first. After performing the operations on the subdividing vertices, the resulting graph is a minor of $(G,\gamma)$. Therefore $(H,\gamma_1)$ is a minor of $(G,\gamma)$.
\end{proof}

\subsection{Rooted (signed) graphs}\label{sec:rootedSignedGraphs}
A \emph{rooted graph} is a pair $(G, K)$ where $G$ is a graph and $K \subseteq V(G)$ denotes the set of \emph{roots}. When we contract an edge in a rooted graph, the resulting vertex is a root if and only if one of the ends of the contracted edge was a root. A rooted graph $(H, K_H)$ is a rooted minor of $(G, K_G)$ if it can be obtained by vertex deletion, edge deletion, contraction, and removing vertices from the set of roots.

A \emph{rooted signed graph} is a triple $(G, \gamma, K)$ where $(G, \gamma)$ is a signed graph and $K \subseteq V(G)$. A rooted signed graph $(H, \gamma_H, K_H)$ is a rooted minor of $(G, \gamma_G, K_G)$ if it can be obtained via vertex deletion, edge deletion, shifting, contracting even edges, and removing vertices from the set of roots.

We note that \cref{lem:totallyevengrid} and \cref{cor:forbidGridForbidEvenGridEquivalence} apply more generally for the rooted grid. Recall from \cref{def:rootedGrid} that $\mathcal{W}_k = (\mathbf{W}_k, \gamma_0, K_{\mathbf{W}_k})$.

\begin{lemma}\label{lem:totallyEvenRootedGrid}
    Let $k \in \mathbb{N}$. Let $(\mathbf{W}_{k^2},K_{\mathbf{W}_{k^2}})$ be the $k^2 \times k^2$ grid rooted at the first row of $\mathbf{W}_{k^2}$. Then $(\mathbf{W}_{k^2}, \gamma, K_{\mathbf{W}_{k^2}})$ contains $\mathcal{W}_k$ as a rooted grid minor for every $\gamma$.
\end{lemma}
\begin{proof}
    The proof is exactly as in \cref{lem:totallyevengrid}, noting that for any grid minor found by the proof, there exists $k$ vertex disjoint paths from the first row of the grid minor to the first row of $\mathbf{W}_{k^2}$.
\end{proof}

\begin{corollary}\label{cor:forbidRootedGridForbidEvenRootedGridEquivalence}
    For any $\gamma$, the following holds. If $(G,K)$ forbids $(\mathbf{W}_k, K_{\mathbf{W}_k})$ as a rooted minor, then $(G,\gamma,K)$ forbids $\mathcal{W}_k$ as a rooted signed graph minor. Conversely, if $(G,\gamma,K)$ forbids $\mathcal{W}_k$ as a rooted signed graph minor, then $(G,K)$ forbids $(\mathbf{W}_{k^2}, K_{\mathbf{W}_{k^2}})$ as a rooted minor.
\end{corollary}

\subsection{Matrices}\label{sec:matrices}

Let $A$ be an integer matrix with exactly two nonzero entries per row. By $G(A)$ we denote the graph whose edge-vertex incidence matrix is the support of $A$. That is, $G(A)$ has vertex set equal to the columns of $A$ and edge set equal to the rows of $A$, and vertex $v$ is incident to edge $e$ if and only if $A_{e,v}$ is nonzero. By $G^+(A)$ we denote the signed graph $(G(A), \gamma)$ where $\gamma(e) = 0$ if the entries of row $e$ have different signs, and $\gamma(e) = 1$ if the entries of row $e$ have the same sign. By $G^+_\bullet(A)$ we denote the rooted signed graph $(G(A), \gamma, K)$ where $(G(A), \gamma) = G^+(A)$ and $K$ is the set of columns that contain an entry outside of $\{-1, 0, 1\}$.

A matrix $A \in \mathbb{Z}^{m \times n}$ is \emph{totally $\Delta$-modular} if every subdeterminant of $A$ is contained in $\{-\Delta, -\Delta + 1, \ldots, \Delta\}$. If $A$ has two nonzero entries per row, this is equivalent to bounding $\|A\|_\infty$, the number of columns with entries outside $\{-1, 0, 1\}$, and the odd cycle packing of the corresponding graph, as given by the lemma below.

\begin{lemma}[Lemma 10.2~\cite{ChoiGKMW2025OCPtw}]\label{lem:deltamodularequivalence}
    Let $A \in \mathbb{Z}^{m \times n}$ have two nonzero entries per row and let $(G, \gamma, K) = G^+_\bullet(A)$. Then if $A$ is totally $\Delta$-modular,
    \begin{enumerate}
        \item $\|A\|_\infty \leq \Delta$,
        \item $|K| \leq 2\log_2\Delta$, and
        \item $\ocp(G,\gamma) \leq \log_2\Delta$.
    \end{enumerate}
    Conversely, if $\Delta$ is the largest absolute value of a subdeterminant of $A$, then $\Delta \leq 2^{\ocp(G,\gamma)}\|A\|_{\infty}^{|K|}$.
\end{lemma}

\section{\texorpdfstring{$K$}{K}-free treewidth}\label{sec:Kfreetw}

Jansen and Swennenhuis~\cite{JansenS2024steiner} introduced $K$-free treewidth as a notion of treewidth specialized for rooted graphs. This notion will be central in decomposing our rooted signed graph with respect to the roots. We note this differs from other notions of treewidth with respect to rooted graphs, such as annotated treewidth or bidimensionality \cite{sau2025parameterizing, protopapas2025colorful}.

\begin{definition}[Tree $K$-free-decomposition]
    Let $G$ be a graph and $K \subseteq V(G)$. A \emph{tree $K$-free-decomposition} of $G$ is a triple $(T,\beta,L)$ where $T$ is a tree, $L \subseteq V(G) \setminus K$, and $\beta:V(T) \rightarrow 2^{V(G)}$ such that
    \begin{enumerate}
        \item $(T,\beta)$ is a tree decomposition of $G$, and
        \item for each $v \in L$, there exists a unique $t \in V(T)$ such that $v \in \beta(t)$, and $t$ is a leaf of $T$.
    \end{enumerate}
    The \emph{width} of a tree $K$-free-decomposition $(T, \beta, L)$ is $\max\{0, \max_{t \in V(T)} |\beta(t) \setminus L| - 1\}$. The \emph{$K$-free treewidth} of a graph $G$, denoted $\mathsf{tw}_K(G,K)$ is the minimum width of a tree $K$-free-decomposition of $G$.
\end{definition}
Intuitively, a graph has low $K$-free treewidth if there exists a tree decomposition where internal bags must be small, but leaf bags can contain many vertices outside of $K$.

Jansen and Swennenhuis~\cite{JansenS2024steiner} also gave an FPT 5-approximation for $K$-free treewidth.

\begin{theorem}[Theorem 3.4~\cite{JansenS2024steiner}]\label{thm:FPTapproxForKfreeTW}
    There exists an algorithm that, given an $n$-vertex graph $G$, $K \subseteq V(G)$, and a non-negative integer $k$, either computes a tree $K$-free-decomposition $(T,\beta)$ of $G$ of width at most $5k+5$ with $|V(T)| \in \mathcal{O}(n)$ nodes, or correctly returns that $\mathsf{tw}_K(G,K) > k$. The algorithm runs in time $2^{\mathcal{O}(k)}\mathsf{poly}(n)$.
\end{theorem}

Marx, Seymour, and Wollan~\cite{MarxSW2017Rooted} proved a grid theorem for rooted minors in the language of tangles, but we present the statement in terms of $K$-free treewidth as given in \cite{HodorLMR2024quickly}.
\begin{theorem}[Theorem 9~\cite{HodorLMR2024quickly}]\label{thm:forbiddingRootedGrid}
    There exists a function $f_{\ref{thm:forbiddingRootedGrid}}: \mathbb{N} \rightarrow \mathbb{N}$ such that for every positive integer $k$, every graph $G$, and every set $K \subseteq V(G)$, if $(G,K)$ does not contain $(\mathbf{W}_k, K_{\mathbf{W}_k})$ as a rooted-minor, then the $K$-free treewidth of $G$ is at most $f_{\ref{thm:forbiddingRootedGrid}}(k)$. Furthermore $f_{\ref{thm:forbiddingRootedGrid}} \in \mathcal{O}(k^{36 + o(1)})$.
\end{theorem}

We also show that the approximate converse is true. First, we note the following observation.

\begin{observation}\label{obs:KfreeTWminormonotone}
    If $(H,K_H)$ is a rooted minor of $(G,K_G)$, then $\mathsf{tw}_K(H,K_H) \leq \mathsf{tw}_K(G,K_G)$.
\end{observation}
\begin{proof}
    Clearly removing vertices from $K_G$ and deleting vertices or edges from $G$ can only decrease the $K$-free treewidth. If $(T,\beta,L)$ is a tree $K_G$-free-decomposition of $(G,K_G)$ and we contract an edge $uv \in E(G)$ to obtain the rooted graph $(H,K_H)$, then by identifying the two vertices in each bag of the tree we obtain a tree $K_H$-free-decomposition of $(H,K_H)$ of only smaller width. The new vertex is in $L$ if and only if both ends of the contracted edge were in $L$.
\end{proof}

We now show the approximate converse of \cref{thm:forbiddingRootedGrid}.

\begin{lemma}\label{lem:rootedGridLargeKFreeTW}
    If $(G,K)$ contains $(\mathbf{W}_{3k+1}, K_{\mathbf{W}_{3k+1}})$ as a rooted minor, then $\mathsf{tw}_K(G,K) \geq k$.
\end{lemma}
\begin{proof}
    By \cref{obs:KfreeTWminormonotone}, it suffices to show that the $K$-free treewidth of $(\mathbf{W}_{3k+1}, K_{\mathbf{W}_{3k+1}})$ is at least $k$. Let $(\mathbf{W}_{3k+1}, K_{\mathbf{W}_{3k+1}}) = (W,K)$. Note that $K$ is highly connected set. That is, for any $A,B \subseteq K$, there are $\min\{|A|,|B|\}$ vertex disjoint $(A,B)$-paths in $W$. Suppose for contradiction there exists a tree $K$-free-decomposition $(T,\beta,L)$ of $(W,K)$ of width less than $k$. We may assume that $T$ has maximum degree at most 3 as follows. For any vertex $t \in V(T)$ with degree larger than 3, create a new vertex $t'$ adjacent to $t$ and make two of the neighbors of $t$ be adjacent to $t'$ instead of $t$. By setting $\beta(t') = \beta(t)$, we obtain a tree $K$-free-decomposition of width less than $k$. By repeating this process we may assume that $T$ has maximum degree at most 3.
    
    Each edge $t_1t_2 \in E(T)$ corresponds to a separation $(G_1,G_2)$ in $G$ of order at most $k$. That is, if $T_1,T_2$ are the components of $T-t_1t_2$ containing $t_1,t_2$ respectively, then $V(G_1) = \bigcup_{t \in V(T_1)} \beta(t)$, $V(G_2) = \bigcup_{t \in V(T_2)} \beta(t)$, and there are no edges in $G$ with one end in $V(G_1) \setminus V(G_2)$ and on end in $V(G_2) \setminus V(G_1)$. Because $K$ is highly connected, there exists a unique side of the $i \in [2]$ such that $G_i$ contains at most $k$ vertices from $K$. We can orient the edges of $T$ to point away from this side, and because $T$ is acyclic, there exists a vertex $t \in V(T)$ which is a sink. Let $T_1,  \dots, T_\ell$ be the subtrees of $T$ whose union is $T$ and such that each $T_i,T_j$ pairwise intersect exactly at $t$. Note that $\ell = d(t) \leq 3$. For $i \in [\ell]$, let $A_i = \bigcup_{t' \in V(T_i)} \beta(t')$. Then $\bigcup_{i \in [\ell]} A_i = V(G)$ and each $A_i$ contains at most $k$ of the vertices of $K$, contradicting that $|K| = 3k+1$.
\end{proof}

\section{(tame) \ocp-treewidth}\label{sec:ocptw}

Odd-Cycle-Packing-treewidth (or \ocp-treewidth) was introduced by Choi, Gorsky, Kim, McFarland, and Wiederrecht~\cite{ChoiGKMW2025OCPtw} as a variant of treewidth concerned with the odd cycle packing number of each bag. For notational convenience we will only be concerned with tame \ocp-tree-decompositions, and we do not define \ocp-tree-decompositions in general. The minimum width of a tame \ocp-tree-decomposition is functionally equivalent to the minimum width of an \ocp-tree-decomposition, see \cite[Theorem~4.2]{ChoiGKMW2025OCPtw}, but the definition of a tame \ocp-tree-decomposition matches more closely with the definition of a \tdm-tree-decomposition.

\begin{definition}[tame \ocp-tree-decomposition]
    A \emph{tame \ocp-tree-decomposition} of a graph $G$ is a triple $(T, \beta, \alpha)$ where $T$ is a tree and $\beta, \alpha \colon V(T) \rightarrow 2^{V(G)}$ such that
    \begin{enumerate}
        \item $(T, \beta)$ is a tree decomposition of $G$,
        \item $\alpha(t) \subseteq \beta(t)$ for each $t \in V(T)$, and
        \item for each $tt' \in E(T)$, $|(\beta(t) \cap \beta(t'))\setminus \alpha(t)| \leq 1$.
    \end{enumerate}
    The \emph{tame width (t-width)} of a tame \ocp-tree-decomposition $(T,\beta,\alpha)$ is:
    \begin{align*}
        \max_{t\in V(T)}|\alpha(t)| +  \ocp(G[\beta(t)\setminus\alpha(t)]).
    \end{align*}
    We note that this definition of width differs by at most 1 from the usual definition of the width for a tame \ocp-tree-decomposition, but is easier to work with for tame decompositions.
    The \emph{\tocp-treewidth} of a graph $G$, denoted $\tocptw(G)$, is the minimum tame width of a tame \ocp-tree-decomposition of $G$.
\end{definition}

We can extend this definition to signed graphs $(G, \gamma)$ by letting the odd cycle packing number in the above definition refer to odd cycle packing number of the signed graph. Note that $\tocptw(G,\gamma_1) = \tocptw(G)$. Similarly we can extend the definition to rooted signed graphs by ignoring the set of roots. Note that for a rooted signed graph $(G,\gamma,\varnothing)$, the sets $\alpha(t)$ in the above definition are protectors for the bags $\beta(t)$. In particular, if $(T,\beta,\alpha,J)$ is a \tdm-tree-decomposition of $(G,\gamma,K)$, then $(T,\beta,\alpha)$ is a tame \ocp-tree-decomposition of $(G,\gamma)$.

Choi et al. noted that (tame) \ocp-treewidth is odd-minor-monotone, but we show more generally that the \tocp-treewidth of signed graphs is minor monotone with respect to signed graph minors.

\begin{lemma}\label{lem:OCPTWsignedGraphMinorMonotone}
    Let $(H,\gamma_H)$ be a minor of $(G,\gamma_G)$. Then $\tocptw(H,\gamma_H) \leq \tocptw(G,\gamma_G)$.
\end{lemma}
\begin{proof}
    Clearly deleting an edge or vertex can only decrease the \tocp-treewidth. Shifting also does not change the \tocp-treewidth because it does not change the odd cycle packing of any subgraph. Thus we may assume that $(H,\gamma_H)$ is obtained from $(G,\gamma_G)$ by contracting a single even edge. Let $(T,\beta,\alpha)$ be a tame \ocp-tree-decomposition of $(G,\gamma_G)$. We then replace both ends of the edge with the new vertex in every $\beta(t),\alpha(t)$ for $t \in V(T)$ to obtain a tame \ocp-tree-decomposition of $(H,\gamma_H)$ of only smaller t-width. We note that contracting an even edge can only decrease the odd cycle packing number of every subgraph.
\end{proof}

The following lemma will be helpful for translating between decompositions for unsigned graphs and signed graphs.

\begin{lemma}\label{lem:SubdivisionsDontChangeOCPTW}
    Let $(G',\gamma')$ be obtained from $(G,\gamma)$ by replacing an edge with a path of the same parity. Then $\tocptw(G',\gamma') = \tocptw(G,\gamma)$. Furthermore there exists an algorithm which takes as input $(G',\gamma')$, $(G,\gamma)$, and a tame \ocp-tree-decomposition of $(G',\gamma')$ of t-width $k$ and returns a tame $\ocp$-tree-decomposition of $(G,\gamma)$ of t-width $k$ with the same number of bags in time $g(k)\mathsf{poly}(|V(G')|)$.
\end{lemma}
\begin{proof}
    Suppose $(G',\gamma')$ is obtained from $(G,\gamma)$ by replacing the edge $xy \in E(G)$ by the path $P = xv_1v_2\dots v_\ell y$ where $\gamma'(e) = \gamma(e)$ for all $e \in E(G) \cap E(G')$ and $\sum_{e \in E(P)} \gamma'(e) = \gamma(xy)$. Let $(T,\beta,\alpha)$ be a tame \ocp-tree-decomposition for $(G,\gamma)$. There is some bag $\beta(t)$ for which $x,y \in \beta(t)$. Let $\beta'(t) = \beta(t) \cup V(P)$. For all $t' \in V(T) \setminus \{t\}$, let $\beta'(t') = \beta(t')$. Then $(T,\beta',\alpha)$ is a tame \ocp-tree-decomposition of $(G',\gamma')$ of the same t-width. Note that $\ocp(G'[\beta'(t) \setminus \alpha(t)],\gamma') = \ocp(G[\beta(t) \setminus \alpha(t)], \gamma)$. Hence $\tocptw(G',\gamma') \leq \tocptw(G,\gamma)$. We note that $(G,\gamma)$ is a minor of $(G',\gamma')$, and so $\tocptw(G,\gamma) \leq \tocptw(G',\gamma')$.

    Because $(G,\gamma)$ is obtained from $(G',\gamma')$ by only shifting and contracting edges in $P$, one can simply identify the internal vertices of $P$ with one of its endpoints in each bag to algorithmically obtain a tame \ocp-tree-decomposition for $(G,\gamma)$ as in the proof of \cref{lem:OCPTWsignedGraphMinorMonotone}.
\end{proof}

Choi et al.~\cite{ChoiGKMW2025OCPtw} also gave an FPT approximation for $\tocp$-treewidth, and they gave an excluded odd-minor characterization for having bounded \tocp-treewidth.

\begin{theorem}[Theorem 4.2~\cite{ChoiGKMW2025OCPtw}]\label{thm:FPTapproxForOCPTW}
    There exists an algorithm that, given an $n$-vertex graph $G$ and a non-negative integer $k$, either computes a tame \ocp-tree-decomposition $(T,\beta,\alpha)$ for $G$ of t-width at most $f_{\ref{thm:FPTapproxForOCPTW}}(k)$ and with $|V(T)| \in \mathcal{O}(n)$, or correctly returns that $\tocptw(G) > k$. The algorithm runs in time $\mathcal{O}_k(n^6)$. Furthermore, $f_{\ref{thm:FPTapproxForOCPTW}}(k) \in \mathsf{poly}(k)$.
\end{theorem}

\begin{theorem}[Theorem 1.6 and Theorem 3.4~\cite{ChoiGKMW2025OCPtw}]\label{thm:OCPtwOddMinorGridThm}
    There exists a polynomial $f_{\ref{thm:OCPtwOddMinorGridThm}}: \mathbb{N} \rightarrow \mathbb{N}$ such that for every positive integer $k$ and every graph $G$,
    \begin{enumerate}
        \item if $G$ contains $\mathbf{H}_k$ or $\mathbf{V}_k$ as an odd-minor, then $\tocptw(G) \geq \frac{k}{2} - 1$, and
        \item if $\tocptw(G) \geq f_{\ref{thm:OCPtwOddMinorGridThm}}(k)$, then $G$ contains $\mathbf{H}_k$ or $\mathbf{V}_k$ as an odd-minor.
    \end{enumerate}
\end{theorem}

The above imply the same results for signed graphs by utilizing \cref{obs:oddMinorSignedGraphEquivalenceMinDeg3} and \cref{lem:SubdivisionsDontChangeOCPTW}.

\begin{corollary}\label{cor:FPTapproxForOCPTWsignedgraphs}
    There exists an algorithm that, given an $n$-vertex, $m$-edge signed graph $(G, \gamma)$ and a non-negative integer $k$, either computes a tame \ocp-tree-decomposition $(T,\beta,\alpha)$ for $(G,\gamma)$ of t-width at most $f_{\ref{thm:FPTapproxForOCPTW}}(k)$ and with $|V(T)| \in \mathcal{O}(n)$, or correctly returns that $\tocptw(G) > k$. The algorithm runs in time $g(k)\mathsf{poly}(m,n)$.
\end{corollary}
\begin{proof}
    Let $G'$ be the graph obtained from $(G,\gamma)$ by subdividing every even edge exactly once. Then by \cref{thm:FPTapproxForOCPTW}, we can compute a tame \ocp-tree-decomposition $(T,\beta,\alpha)$ of $G'$ of t-width at most $f_{\ref{thm:FPTapproxForOCPTW}}(k)$ and with $|V(T)| \in \mathcal{O}(|V(G')|) = \mathcal{O}(n + m)$, or correctly return that $\tocptw(G,\gamma) = \tocptw(G',\gamma_1) = \tocptw(G') > k$ by \cref{lem:SubdivisionsDontChangeOCPTW}. Note that $(T,\beta,\alpha)$ is a tame \ocp-tree-decomposition of $(G',\gamma_1)$ of t-width $f_{\ref{thm:FPTapproxForOCPTW}}(k)$. Then by \cref{lem:SubdivisionsDontChangeOCPTW}, we can compute a tame \ocp-tree-decomposition of $(G,\gamma)$ of the same t-width.
    
    We can augment the decomposition such that $|V(T)| \in \mathcal{O}(n)$ as follows. For any edge $t_1t_2 \in E(T)$ such that $\beta(t_1) \subseteq \beta(t_2)$, we contract the edge $t_1t_2$ in $T$ to create the new vertex $t$. We then set $\beta(t) = \beta(t_2)$ and $\alpha(t) = \alpha(t_2)$. Note that for any vertex $t'$ which was originally adjacent to $t_1$ in $T$, $\beta(t_1) \cap \beta(t') = \beta(t_2) \cap \beta(t')$, and so the resulting decomposition still satisfies $|(\beta(t) \cap \beta(t')) \setminus \alpha(t)| \leq 1$ for all $t'$ adjacent to $t$. The resulting decomposition has $|V(T)| \in \mathcal{O}(n)$. To see this, root $T$ at an arbitrary vertex $r$. It follows that for any vertex $t \in V(T)$ which is not the root, the bag of $t$ contains a new vertex $v$ not contained in the bag of $t$'s parent such that $\{x \in V(T) : v \in \beta(x)\}$ is contained in the subtree rooted at $t$. This defines an injection from $V(T) \setminus \{r\}$ to $V(G)$, and so $|V(T)| \in \mathcal{O}(n)$.
\end{proof}

\begin{corollary}\label{cor:OCPtwSignedGraphGridThm}
    There exists a polynomial $f_{\ref{thm:OCPtwOddMinorGridThm}}: \mathbb{N} \rightarrow \mathbb{N}$ such that for every positive integer $k$ and every signed graph $(G, \gamma)$,
    \begin{enumerate}
        \item if $(G,\gamma)$ contains $(\mathbf{H}_k, \gamma_1)$ or $(\mathbf{V}_k, \gamma_1)$ as a signed graph minor, then $\tocptw(G, \gamma) \geq k/2-1$, and
        \item if $\tocptw(G,\gamma) \geq f_{\ref{thm:OCPtwOddMinorGridThm}}(k)$, then $(G,\gamma)$ contains $(\mathbf{H}_k,\gamma_1)$ or $(\mathbf{V}_k,\gamma_1)$ as a signed graph minor.
    \end{enumerate}
\end{corollary}
\begin{proof}
    Let $G'$ be the graph obtained from $(G,\gamma)$ by subdividing every even edge exactly once. Then by \cref{lem:SubdivisionsDontChangeOCPTW}, $\tocptw(G') = \tocptw(G',\gamma_1) = \tocptw(G,\gamma)$. The result then follows immediately from \cref{thm:OCPtwOddMinorGridThm} and \cref{obs:oddMinorSignedGraphEquivalenceMinDeg3}.
\end{proof}

Fiorini et al.~\cite{Fiorini2025Integer} gave a connection between solving integer programs over matrices in $\{-1, 0, 1\}^{m \times n}$ with two nonzero entries per row and solving \mis.  Choi et al.~\cite{ChoiGKMW2025OCPtw} used this to solve integer programs with coefficient matrix $A$ such that the associated signed graph $G^+(A)$ has bounded (tame) \ocp-treewidth.


\begin{theorem}[Theorem 1.3~\cite{ChoiGKMW2025OCPtw}]\label{thm:IPsolutionForForbiddingParityGrids}
    Let $A \in \{-1, 0, 1\}^{m \times n}$ have two nonzero entries per row such that the associated signed graph has \tocp-treewidth at most $k$. Then there exists a computable function $f_{\ref{thm:IPsolutionForForbiddingParityGrids}}(k)$ such that we can solve the integer program
    \[\max\{w^{\text{T}} x : Ax \leq b, x \in [\ell,u] \cap \mathbb{Z}^n\}\]
    in $n^{f_{\ref{thm:IPsolutionForForbiddingParityGrids}}(k)}$ time, where $w \in \mathbb{Z}^n$, $b \in \mathbb{Z}^m$, and $\ell,u \in (\mathbb{Z} \cup \{-\infty, \infty\})^n$.
\end{theorem}

Our main algorithmic result \cref{thm:MainAlgorithm} is an extension of \cref{thm:IPsolutionForForbiddingParityGrids} to when $A$ has entries outside of $\{-1,0,1\}$.

\section{Properties of \tdm-treewidth}\label{sec:PropertiesTDMTW}

In this section we show how \tdm-treewidth is a combination of $K$-free treewidth and \tocp-treewidth. We use these connections to prove \cref{thm:MainGridTheorem} and give an FPT-approximation for the \tdm-treewidth of a graph. We later use these results to prove \cref{thm:MainAlgorithm} in \cref{sec:SolvingIP}.

First, we show that \tdm-treewidth is minor monotone with respect to rooted signed graphs.

\begin{lemma}
    If $(H,\gamma_H,K_H)$ is a minor of $(G,\gamma_G,K_G)$, then $\tdmtw(H,\gamma_H,K_H) \leq \tdmtw(G,\gamma_G,K_G)$.
\end{lemma}
\begin{proof}
    We may assume that $(H,\gamma_H,K_H)$ is formed from $(G,\gamma_G,K_G)$ be via a single vertex deletion, a single edge deletion, shifting at a single vertex, contracting one even edge, or removing a vertex from the set of roots. Let $(T,\beta,\alpha,J)$ be a \tdm-tree-decomposition for $(G,\gamma_G,K_G)$. If $(H,\gamma_H,K_H)$ is formed by an edge deletion, shifting, or removing a vertex from the set of roots, then $(T,\beta,\alpha,J)$ is a \tdm-tree-decomposition for $(H,\gamma_H,K_H)$ of the same width. Note that shifting does not change the odd cycle packing number of any subgraph. If $(H,\gamma_H,K_H)$ is formed by vertex deletion, then by removing that vertex from the image of $\beta(t),\alpha(t)$ for every $t \in V(T)$, we obtain a \tdm-tree-decomposition for $(H,\gamma_H,K_H)$ of only smaller width. Suppose that $(H,\gamma_H,K_H)$ is formed by contracting the edge $uv$ to create the new vertex $v'$. We then set replace each instance of $u$ or $v$ in any of $\beta(t),\alpha(t)$ with $v'$ in order to obtain a \tdm-tree-decomposition of $(H,\gamma_H,K_H)$ of only smaller width. Note that contracting an even edge can only decrease the odd cycle packing number.
\end{proof}

We now more explicitly state the relation between the width parameters discussed thus far.

\begin{lemma}\label{lem:equivalenceOfDecompositions}
    Let $(G,\gamma,K)$ be a rooted signed graph. Then we have the following relationship between parameters.

    \[\max\{\mathsf{tw}_K(G,K)+1, \tocptw(G,\gamma)\} \leq \tdmtw(G,\gamma,K) \leq \mathsf{tw}_K(G,K) + 1 + \tocptw(G,\gamma)\]
    

\end{lemma}
\begin{proof}
    Let $(J,\beta_J,L)$ be a tree $K$-free-decomposition of $(G,K)$ of width at most $k-1$ and suppose that $\tocptw(G,\gamma) = r$. For every leaf $j$ of $J$, $\tocptw(G[\beta_J(j) \cap L],\gamma) \leq \tocptw(G,\gamma) = r$. Let $\ell(J)$ denote the set of leaves of $J$. For each $j \in \ell(J)$, let $(T_j,\beta_j,\alpha_j)$ be a tame \ocp-tree-decomposition of $(G[\beta_J(j) \cap L], \gamma)$ of t-width at most $r$. We now construct the \tdm-tree-decomposition of $(G,\gamma,K)$ as follows. Let $T$ be a tree formed from the disjoint union of $J$ and $T_j$ for $j \in \ell(J)$ by adding an edge from each leaf $j \in \ell(J)$ to an arbitrary vertex in $T_j$. For each $j \in V(J)$, let $\beta(j) = \alpha(j)= \beta_J(j) \setminus L$. For each $t \in V(T_j)$, let $\beta(t) = \beta_j(t) \cup (\beta_J(j) \setminus L)$ and $\alpha(t) = \alpha_j(t) \cup (\beta_J(j) \setminus L)$. Then $(T,\beta,\alpha,J)$ is a \tdm-tree-decomposition of $(G,\gamma,K)$ of width at most $k + r$.

    Suppose now that $(T,\beta,\alpha,J)$ is a \tdm-tree-decomposition of $(G,\gamma,K)$ of width at most $k$. Then $(T,\beta,\alpha)$ is a tame \ocp-tree-decomposition of $(G,\gamma)$ of t-width at most $k$, so it suffices to show $\mathsf{tw}_K(G,K) \leq k-1$. We augment $(T,\beta,\alpha,J)$ as follows. For every $j \in V(J) \setminus \ell(J)$ with $\beta(j) \neq \alpha(j)$, we add a new leaf $t$ to $J$ (and thus $T$) adjacent only to $j$ such that $\beta(t) = \beta(j)$ and $\alpha(t) = \alpha(j)$. We then set $\beta(j)$ to be equal to $\alpha(j)$. After performing this augmentation we still have a \tdm-tree-decomposition of width at most $k$ but now with the extra property that $\beta(j) = \alpha(j)$ for every $j \in V(J) \setminus \ell(J)$. We then take $L = V(G) \setminus \bigcup_{j \in V(J)} \alpha(j)$. Root $T$ at an arbitrary vertex in $V(J)$. For $j \in \ell(J)$, let $T_j$ be the subtree of $T$ consisting of all vertices whose unique path to the root includes $j$ and let $\beta'(j) = \bigcup_{t \in V(T_j)} \beta(t)$. For $j \in V(J) \setminus \ell(J)$, let $\beta'(j) = \beta(j)$. Then $(J,\beta',L)$ is a tree $K$-free-decomposition of $(G,K)$ of width at most $k-1$.
\end{proof}

The above proof naturally gives us an FPT-approximation for $\tdmtw(G,\gamma,K)$.

\begin{theorem}\label{thm:FPTapproxForTDMTW}
    There exists an algorithm that, given an $n$-vertex, $m$-edge rooted signed graph $(G,\gamma,K)$ and a non-negative integer $k$, either computes a \tdm-tree-decomposition $(T,\beta,\alpha,J)$ of $(G,\gamma,K)$ of width at most $f_{\ref{thm:FPTapproxForTDMTW}}(k)$ with $|V(T)| \in \mathcal{O}(n)$, or correctly returns that $\tdmtw(G,\gamma,K) > k$. The algorithm runs in time $g(k)\mathsf{poly}(m,n)$. Furthermore $f_{\ref{thm:FPTapproxForTDMTW}}(k) \in \mathsf{poly}(k)$.
\end{theorem}
\begin{proof}
    We first apply \cref{thm:FPTapproxForKfreeTW} to either obtain a tree $K$-free-decomposition $(J,\beta_J,L)$ of $(G,K)$ of width at most $5k$ with $|J| \in \mathcal{O}(n)$, or we correctly determine that $\mathsf{tw}_K(G,K) > k - 1$. By \cref{lem:equivalenceOfDecompositions} we have $\tdmtw(G,\gamma,K) \geq \mathsf{tw}_K(G,K) + 1$, and so in the second case we may correctly return that $\tdmtw(G,\gamma,K) > k$.
    
    Hence we may assume that we obtain the tree $K$-free-decomposition $(J,\beta_J,L)$. Then for each leaf $j \in V(J)$, we apply \cref{cor:FPTapproxForOCPTWsignedgraphs} to either obtain a tame \ocp-tree-decomposition of $(G[\beta(j) \cap L],\gamma)$ of width at most $f_{\ref{thm:FPTapproxForOCPTW}}(k)$, or we correctly determine that $\tocptw(G[\beta(j) \cap L], \gamma) > k$. Again in the second case we correctly determine that $\tdmtw(G,\gamma,K) \geq \tocptw(G,\gamma) \geq \tocptw(G[\beta(j) \cap L],\gamma) > k$ by \cref{lem:equivalenceOfDecompositions}. Thus we may assume that we obtain a tame \ocp-tree-decomposition $(T_j,\beta_j,\alpha_j)$ of $(G[\beta(j) \cap L], \gamma)$ of width at most $k$ and with $|V(T_j)| \in \mathcal{O}(|\beta(j) \cap L|)$. The number of vertices over all such trees $T_j$ is in $\mathcal{O}(|L|)$. We then construct a \tdm-tree-decomposition $(T,\beta,\alpha,J)$ of $(G,\gamma,K)$ as in the proof of \cref{lem:equivalenceOfDecompositions}. The resulting decomposition has width at most $5k+1 + f_{\ref{thm:FPTapproxForOCPTW}}(k)$ and $|V(T)| \in \mathcal{O}(n)$.
\end{proof}

\cref{lem:equivalenceOfDecompositions} and \cref{cor:forbidRootedGridForbidEvenRootedGridEquivalence} immediately imply \cref{thm:MainGridTheorem}, which we restate below for convenience.

\MainGridTheorem*
\begin{proof}
    Suppose that $(G,\gamma,K)$ forbids $\mathcal{H}_k, \mathcal{V}_k$, and $\mathcal{W}_k$ as a minor. By \cref{cor:forbidRootedGridForbidEvenRootedGridEquivalence}, $(G,K)$ forbids $(\mathbf{W}_{k^2},K_{\mathbf{W}_{k^2}})$ as a rooted minor. Then by \cref{thm:forbiddingRootedGrid}, $\mathsf{tw}_K(G,K) \leq f_{\ref{thm:forbiddingRootedGrid}}(k^2)$. Similarly by \cref{cor:OCPtwSignedGraphGridThm}, $\tocptw(G,\gamma) \leq f_{\ref{thm:OCPtwOddMinorGridThm}}(k)$. Hence the result follows from \cref{lem:equivalenceOfDecompositions}.

    Suppose instead that $(G,\gamma,K)$ contains one of $\mathcal{H}_k, \mathcal{V}_k$, or $\mathcal{W}_k$ as a minor. If $(G,\gamma,K)$ contains $\mathcal{W}_k$ as a minor, then $(G,K)$ contains $(\mathbf{W}_k, K_{\mathbf{W}_k})$ as a minor. Then by \cref{lem:rootedGridLargeKFreeTW} and \cref{lem:equivalenceOfDecompositions}, $\tdmtw(G,\gamma,K) \geq \mathsf{tw}_K(G) + 1 \geq \Omega(k)$. If $(G,\gamma,K)$ contains one of $\mathcal{H}_k, \mathcal{V}_k$ as a minor, then $(G,\gamma)$ contains one of $(\mathbf{H}_k,\gamma_1)$, $(\mathbf{V}_k, \gamma_1)$ as a signed graph minor. Thus by \cref{cor:OCPtwSignedGraphGridThm} and \cref{lem:equivalenceOfDecompositions} we have $\tdmtw(G,\gamma,K) \geq \tocptw(G,\gamma) \geq \Omega(k)$.
\end{proof}

\section{Solving integer programs with large entries}\label{sec:SolvingIP}

In this section we prove \cref{thm:MainAlgorithm}, which seeks to extend \cref{thm:IPsolutionForForbiddingParityGrids} to matrices with entries outside $\{-1, 0, 1\}$ under the condition that variables have bounded domain size, and there is no ``highly linked'' set of columns with entries outside $\{-1, 0, 1\}$. We restate \cref{thm:MainAlgorithm} below for convenience.

\MainAlgorithmicTheorem*
\begin{proof}
    Let $G^+_\bullet(A) = (G,\gamma,K)$. For each $v \in V(G)$, let $w_v,x_v,\ell_v,u_v$ be the associated entries in $w,x,\ell,u$ respectively. By \cref{lem:equivalenceOfDecompositions}, $\mathsf{tw}_K(G,K) \leq k-1$. Then by \cref{thm:FPTapproxForKfreeTW}, we obtain a tree $K$-free-decomposition $(T,\beta,L)$ of $(G,K)$ of width at most $5k$ with $|V(T)| \in \mathcal{O}(n)$ in time $2^{\mathcal{O}(k)}\mathsf{poly}(n)$. We may assume that $T$ contains a non-leaf vertex, otherwise we subdivide the single edge in $T$ to create a new vertex $t$ and set $\beta(t) = \beta(t') \setminus L$ for one of the other vertices $t'$ of $T$. We root $T$ at an arbitrary non-leaf vertex $r$. For each $t \in V(T)$, we denote by $T_t$ the subtree of $T$ consisting of all vertices whose unique path to the root includes $t$. We denote by $\mathsf{IP}_t$ the above integer program restricted only to variables $x_v$ for $v \in \bigcup_{t' \in V(T_t)} \beta(t')$. We keep the constraints which include only the variables we restricted to.
    
    We construct a dynamic programming table $p[t,\eta_t]$ where for each $t \in V(T)$ and each assignment $\eta_t: \{x_v : v \in \beta(t) \setminus L\} \rightarrow \mathbb{Z}$ of $x_v$ to a value in $[\ell_v, u_v]$, $p[t,\eta_t]$ is the optimal objective value of $\mathsf{IP}_t$ under the partial assignment $\eta_t$, or we store $-\infty$ if the partial assignment can not be extended to a feasible solution. We note that $p[t,\eta_t]$ is finite if the partial assignment can be extended to a feasible solution. We compute $p[t,\eta_t]$ from the leaves of $T$ towards the root. The final solution is given by
    \[\max_{\eta_r} p[r,\eta_r].\]
    As there are $d^{5k}$ choices for $\eta_r$, this can be computed in time $\mathcal{O}(d^{5k})$ once we fill in the table $p$. We now show how to compute $p[t,\eta_t]$.

    Fix a vertex $t \in V(T)$ and fix $\eta_t$. If $t$ is a leaf, then because
    \[\tocptw(G[\beta(t) \cap L], \gamma) \leq \tocptw(G,\gamma) \leq \tdmtw(G,\gamma,K) \leq k\]
    by \cref{lem:equivalenceOfDecompositions}, we can compute $p[t,\eta_t]$ in $n^{f_{\ref{thm:IPsolutionForForbiddingParityGrids}}(k)}$ time by \cref{thm:IPsolutionForForbiddingParityGrids}.
    
    We may then assume that $t$ is not a leaf and we have computed $p[t',\eta_{t'}]$ for every $t'$ a child of $t$ and every $\eta_{t'}$. For each child $t'$ of $t$, we compute
    \[s[t',\eta_t] = \max\left\{p[t',\eta_{t'}] - \sum_{v \in \beta(t') \cap \beta(t)} w_v\eta_{t'}(x_v) : \eta_{t'}(x_v) = \eta_t(x_v)\text{ for all }v \in \beta(t) \cap \beta(t')\right\}.\]
    We note that $s[t',\eta_t]$ is uniquely determined by $\eta_t$ restricted to $\beta(t) \cap \beta(t')$. Thus we can compute $s[t',\eta_t]$ for all $\eta_t$ when we compute $p[t',\eta_{t'}]$ for all $\eta_{t'}$ in time $\mathcal{O}_k(d^{5k})$. This can be done by iterating over all choices of $\eta_t$ restricted to $\beta(t) \cap \beta(t')$, and then for each choice iterating over all ways to extend it to $\eta_{t'}$ by choosing the assignment of $x_v$ for $v \in \beta(t') \setminus \beta(t)$.
    
    We then set
    \[p[t,\eta] = \sum_{v \in \beta(t)} w_v\eta(x_v) + \sum_{t'\text{ child of }t} s[t',\eta].\]
    if the assignment $\eta$ does not violate any constraint induced by the bag of $t$, otherwise we set $p[t,\eta] = -\infty$. The correctness of the above algorithm follows immediately by induction.
\end{proof}

\section{Further discussion}\label{sec:futherDiscussion}
In this section we discuss some of the assumptions made in various theorems and whether they can be relaxed.

\subsection{On \texorpdfstring{$\{0,1\}$}{\{0,1\}}-variables and the definition of \tdm-treewidth}
The definition of a \tdm-tree-decomposition is rather technical and relies on specific properties of the adhesions. Perhaps it would be more natural to require all adhesions to have some bounded size outside of the protector set, or to not require some protectors to be strong. The reason for having an adhesion of size 1 outside of a hitting set is due to $\{0,1\}$-variables. Suppose all variables take values in $\{0,1\}$, and we are doing dynamic programming on the tree decomposition. At some bag $\beta(t)$ for $t \in V(T)$, we guess the value of variables in $\alpha(t)$ and the value of the adhesion with the parent bag. However there may be an unbounded number of children bags, and for each child $t'$ we must decide the value to take on $(\beta(t) \cap \beta(t')) \setminus \alpha(t)$. If this set has size 1, say it corresponds to variable $x_v$, then there's a nice trick: add a new variable $y_v \in \{0,1\}$, add the constraint $x_v + y_v \leq 1$, set $w(x_v)$ equal to the optimal value of the subtree with root $t'$ if $x_v = 1$, and set $w(y_v)$ equal to the optimal value of the subtree if $x_v = 0$. We may assume $w(x_v), w(y_v) \geq 0$ by adding a constant to the final weight. The optimal solution will then set exactly one of $x_v$ or $y_v$ to 1, and adding this constraint doesn't affect that the resulting problem is totally $\Delta$-modular. Hence we can make this construction for an unbounded number of child bags. If instead $x_v$ can take more than two values, or $|(\beta(t) \cap \beta(t')) \setminus \alpha(t)| > 1$, it is unclear how to encode this in the integer program without possibly increasing $\Delta$.

The reason we can allow $|(\beta(t) \cap \beta(t')) \setminus \alpha(t)| = 1$ for some bags in our \tdm-tree-decomposition is because we can reduce to $\{0,1\}$-variables for some of the subproblems during the dynamic programming. In the proof of \cref{thm:MainAlgorithm}, we applied \cref{thm:IPsolutionForForbiddingParityGrids} as a black box. What's really happening is that whenever we're in a subtree disjoint from $J$, we can guess on $\alpha(t)$ such that the remaining integer program corresponding to the subtree has coefficients in $\{-1,0,1\}$. We can then solve the LP-relaxation for the entire subtree and use proximity results to make all variables take values in $\{0,1\}$ (see \cite[Theorem 10.3]{ChoiGKMW2025OCPtw} based on \cite[Section 3]{Fiorini2025Integer}). Then we can use the above trick to fill the dynamic programming table even when $t$ has an unbounded number of children with adhesions of size 1.

The above discussion implies, for instance, that if we know that all variables are in $\{0,1\}$, we could solve the integer program even when no protectors are strong. This would correspond to a tame \ocp-tree-decomposition $(T,\beta,\alpha)$ such that $K \cap \beta(t) \subseteq \alpha(t)$ for all $t \in V(T)$ because we may now take $J = T$. This leads to the following question, which asks for a grid theorem for such decompositions.

\begin{question}
    What are the forbidden minors for the class of rooted signed graphs $(G,\gamma,K)$ which have a tame \ocp-tree-decomposition $(T,\beta,\alpha)$ of bounded width such that $K \cap \beta(t) \subseteq \alpha(t)$ for all $t \in V(T)$?
\end{question}

\subsection{Bounded entries and unbounded variables}

In \cref{thm:MainAlgorithm} we assume that $\|u - \ell\|_\infty$ is bounded, but we allow $\|A\|_\infty$ to be unbounded. Bounding $\|A\|_\infty$ is perhaps more natural from the perspective of totally $\Delta$-modular matrices, as this would imply that the bags of the \tdm-tree-decomposition are totally $\Delta$-modular. This leads to the following question.

\begin{question}
    Fix integers $d,k \geq 0$. Let $A$ be a matrix with two nonzero entries per row such that the associated rooted signed graph has \tdm-treewidth at most $k$. Let $w \in \mathbb{Z}^n, b \in \mathbb{Z}^m$, and $\ell,u \in (\mathbb{Z} \cup \{-\infty, \infty\})^n$. If $\|A\|_\infty \leq d$, can we solve the integer program
    \[\max\{w^Tx : Ax \leq b, x \in [\ell,u] \cap \mathbb{Z}^n\}\]
    in polynomial time for constant $d,k$?
\end{question}

In fact, something much simpler is not known.

\begin{conjecture}\label{conj:IPonaTree}
    Fix a positive integer $d$. Let $A$ be a matrix with exactly two nonzero entries per row such that $G(A)$ is a tree and $\|A\|_\infty \leq d$. Then the integer program
    \[\max\{w^Tx : Ax \leq b, x \in [\ell,u] \cap \mathbb{Z}^n\}\]
    can be solved in polynomial time for constant $d$.
\end{conjecture}

If the tree $G(A)$ has constant depth, the \cref{conj:IPonaTree} is implied by results on bounded treedepth (see, for instance, \cite{eisenbrand2019algorithmic}). However \cref{conj:IPonaTree} is not known even for $G(A)$ a path. The path case is perhaps particularly interesting due to the recent result of Bria{\'n}ski et al.~\cite{brianski2025integer} which shows that integer programming is NP-hard even when every non-zero coefficient appears in at most two consecutive constraints and $\|A\|_\infty \leq 8$. The resulting matrix has a ``path-like'' structure, but crucially may contain 3 nonzero entries per row. Due to the lack of hardness results for integer programs with at most two nonzero entries per row, it is not clear whether we must bound one of $\|u - \ell\|_\infty$ or $\|A\|_\infty$. It could be that \cref{conj:IPonaTree} is true without the bound on $\|A\|_\infty$.

Another interesting question is whether the rooted grid in \cref{cor:MainAlgorithmGridVersion} must be forbidden. In fact, it is unclear whether a grid with every vertex in the set of roots must be forbidden. That is, is there a hardness result for integer programs where large entries are everywhere, but the corresponding signed graph has no odd cycles?

\begin{question}
    Is it NP-hard to solve the integer program
    \[\max\{w^Tx : Ax \leq b, x \in [\ell,u] \cap \mathbb{Z}^n\}\]
    over instances where $G^+_\bullet(A) = (G,\gamma,K)$ is planar, has no odd cycles, but $K$ is arbitrary?
\end{question}

The above question is unknown even for $G^+_\bullet(A)$ not necessarily planar. Any hardness result for integer programs with at most two nonzero entries per row that isn't the standard reduction to \mis would be interesting.\\
\\
\noindent
\textbf{Acknowledgments.}
The author would like to thank Maximilian Gorsky, Rose McCarty, and Sebastian Wiederrecht for valuable discussions.

\bibliographystyle{alphaurl}
\bibliography{newbib}

\end{document}